\documentclass[12pt,a4paper]{article}
\usepackage{amsmath}
\usepackage{amsfonts}
\usepackage{amsthm}
\usepackage[varg]{txfonts}
\usepackage{graphicx}
\usepackage{amscd}
\usepackage{color}

\numberwithin{equation}{section}
\numberwithin{figure}{section}

\date{}

\newtheorem{thm}{Theorem}
  
\newtheorem{thrm}{Theorem}[section]
\newtheorem{prop}[thrm]{Proposition}
\newtheorem{lem}[thrm]{Lemma}
\newtheorem{cor}[thrm]{Corollary}

\theoremstyle{remark}

\theoremstyle{definition}
\newtheorem*{dfn}{Definition}
\newtheorem{df}[thrm]{Definition}

\renewcommand{\epsilon}{\varepsilon}
\newcommand{\R}{\mathbb{R}}
\newcommand{\C}{\mathbb{C}}
\newcommand{\lft}[1]{\left#1}
\newcommand{\rgt}[1]{\right#1}
\newcommand{\inter}{\operatorname{int}}
\newcommand{\rd}{\partial}

\newcommand{\uvv}[1]{\partial/\partial #1}

\newcommand{\xist}{\xi_\textup{std}}
\newcommand{\alphast}{\alpha_\textup{std}}
\newcommand{\lambdast}{\lambda_{\textup{can}}}
\newcommand{\op}[1]{\mathcal{O}p (#1)}
\newcommand{\contot}[3]{\mathfrak{Cont}_{\textup{ot}} (#1;#2,#3)}
\newcommand{\cont}[3]{\mathfrak{cont} (#1;#2,#3)}
\newcommand{\vphi}{\varphi}
\newcommand{\sn}{\operatorname{CSN}}
\newcommand{\skor}{\perp'}
\newcommand{\ups}{U_\textup{PS}}
\newcommand{\omtw}{\omega_\texttt{tw}}
\newcommand{\lan}{\left\langle}
\newcommand{\ran}{\right\rangle}
\newcommand{\ot}{\textup{ot}}
\newcommand{\mps}{\mathcal{P}_{T^{n-1}}}

\begin{document} 
%
%
\title{Plastikstufe with toric core}
\author{Jiro ADACHI}
\maketitle

\renewcommand{\thefootnote}{\fnsymbol{footnote}}%
\footnote[0]{This work was supported 
  by JSPS KAKENHI Grant Number 25400077.}%

\begin{abstract} 
  Plastikstufes and overtwistedness for higher-dimensional contact manifolds 
are studied in this paper. 
 It is proved that a contact structure is overtwisted 
if and only if there exists a small plastikstufe with toric core 
that has trivial rotation.
\end{abstract} 

\section{Introduction}\label{sec:intro}
  ``Overtwisted'' is a remarkable class of contact structures 
where a parametric \textit{h}-principle holds. 
 This class for contact structures on higher-dimensional manifolds 
is introduced recently (see~\cite{boelmu}) 
although such a class is introduced for contact structures 
on $3$-dimensional manifolds few decades ago (see~\cite{eliash89}). 
 However, compared with the $3$-dimensional case, 
geometric characterization of this notion is still unclear. 
 There are some discussions and proposals in~\cite{boelmu} and~\cite{camupr}. 
 In this paper, a small improvement of one of the ideas in~\cite{camupr} 
is given.
 The improvement is suitable for the modification of contact structure 
introduced in~\cite{art21}. 

  Contact structure is a hyperplane field on an odd-dimensional manifold 
which is completely non-integrable. 
 Borman, Eliashberg, and Murphy proved the existence 
of a class of contact structures 
where contact structures homotopic to each other as almost contact structures 
are isotopic (see~\cite{boelmu}, \cite{eliash89}). 
 A contact structure in the class is said to be \emph{overtwisted}. 
 The class is defined or characterized 
by the existence of a certain piecewise smooth $2n$-dimensional disc 
embedded into a $(2n+1)$-dimensional contact manifold. 
 However the definition of such a disc is rather complex 
although it is comparatively easy in dimension~$3$ 
(see Subsections~\ref{sec:plastikstufe} and~\ref{sec:overtwisted}). 

  There are some proposals for characterizations of the overtwistedness. 
 A characterization by plastikstufe is given 
by Casals, Murphy, and Presas~\cite{camupr}. 
 A plastikstufe is introduced by Niederkr\"uger~\cite{niederkruger06} 
as an obstruction to symplectic fillability. 
 It is known that 
if a contact structure on a closed $3$-manifold is overtwisted 
then it never appears as a certain boundary 
of a compact symplectic $4$-manifold. 
 In this sense, a plastikstufe is regarded 
as a higher-dimensional generalization of an overtwisted $2$-dimensional disc. 

  In~\cite{camupr}, they impose some conditions on plastikstufes. 
 A \emph{plastikstufe} $\mathcal{P}$ 
is a certain product $D_\ot^2\times B^{n-1}$ 
of a simple overtwisted $2$-disc $D_\ot^2$ 
with a closed orientable $(n-1)$-dimensional manifold $B^{n-1}$ 
embedded into a $(2n+1)$-dimensional contact manifold $(M,\xi)$ 
(see Subsection~\ref{sec:plastikstufe} for definition). 
 This manifold $B^{n-1}$ is called the \emph{core}\/ of the plastikstufe. 
 By definition, a submanifold of $\mathcal{P}$ corresponding to 
$(0,1)\times B^{n-1}\subset D_\ot^2\times B^{n-1}$ is Legendrian, 
where $(0,1)\subset D_\ot^2$ corresponds to a part of a non-compact leaf 
of singular foliation on $D_\ot^2$ (see Figure~\ref{fig:3otdiscs}). 
 It is called a \emph{leaf ribbon}\/ of $\mathcal{P}$. 
 A plastikstufe $\mathcal{P}\subset(M,\xi)$ is said to be \emph{small}\/ 
if it is contained in an open ball in $(M,\xi)$. 
 When a plastikstufe $\mathcal{P}$ 
is small and with spherical core $B^{n-1}=S^{n-1}$, 
it is said to have the \emph{trivial rotation}\/ 
if a leaf ribbon of $\mathcal{P}$ is isotopic to a punctured Legendrian disc 
$\operatorname{int} D^n\setminus\{0\}$. 

  Then the following is given in~\cite{camupr}
(The ``only if'' part is due to the \textit{h}-principle in~\cite{boelmu}): 
%
%
\begin{thrm}[Casals, Murphy, Presas]
    Let $(M,\xi)$ be a contact manifold of dimension $2n+1$. 
 The contact structure $\xi$ is overtwisted if and only if 
there exists a small plastikstufe with a spherical core $S^{n-1}$ 
that has trivial rotation.
\end{thrm} 

  Then our interest goes to examples or constructions of such structures 
on a given manifold. 
 In dimension~$3$, there exists a famous modification of contact structure, 
the Lutz twist, that creates overtwisted discs 
without changing the given manifold. 
 A higher-dimensional generalization of the Lutz twist 
was introduced by Etnyre and Pancholi~\cite{etpa}, that creates a plastikstufe. 
 It is confirmed in~\cite{camupr} that the plastikstufe is small 
and has trivial rotation. 

  On the other hand, another higher-dimensional generalization 
of the Lutz twist is introduced 
by the author~\cite{art21}. 
 By the modification, we also obtain plastikstufes in the given manifold. 
 However the plastikstufes are with toric cores $B^{n-1}=T^{n-1}$. 
 This is the first motivation of this paper. 
 We introduce, in this paper, rotation class of small plastikstufes 
with toric core. 
 Then we prove the following: 
%
%
\begin{thm}\label{thma}
  Let $(M,\xi)$ be a contact manifold of dimension $2n+1$. 
 The contact structure $\xi$ is overtwisted if and only if 
there exists a small plastikstufe with a toric core $T^{n-1}$ 
that has trivial rotation.
\end{thm} 
\noindent
 We should remark that it is proved in~\cite{art21} 
that the modification creates overtwisted discs directly. 

  This paper is organized as follows. 
 In the next section, we review the important notions: 
plastikstufe, loose Legendrian submanifold, and overtwistedness. 
 In Section~\ref{sec:rotPS}, we discuss on rotation class of plastikstufes 
with toric core. 
 Then Theorem~\ref{thma} is proved in Section~\ref{sec:proof}. 
 Last of all, in Section~\ref{sec:gltw}, we discuss how to create plastikstufes 
with toric core and trivial rotation without changing the underlying manifold. 

\smallskip

  Recently, a stronger result is informed by Huang~\cite{huang}. 
 It is claimed that any embedded plastikstufe implies overtwistedness. 
 However, it seems it is still important 
to know concrete shapes of plastikstufes. 

\section{Preliminaries}\label{sec:prelim}
  In this section, we review some notions and properties 
needed in the following sections. 
 Plastikstufe (Subsection~\ref{sec:plastikstufe}), 
loose Legendrian submanifold (Subsection~\ref{sec:looselegendrian}), 
and overtwistedness (Subsection~\ref{sec:overtwisted}) are introduced. 
 In Subsection~\ref{sec:h-princ}, we review some results 
concerning \textit{h}-principle. 

\subsection{Plastikstufe}\label{sec:plastikstufe}
  Plastikstufe is an obstruction to symplectic fillability 
introduced by Niederkr\"uger~\cite{niederkruger06}. 
 A prototype of this notion was introduced by Gromov~\cite{gromov85}. 
 In this subsection, we review the definition and basic properties. 

  In order to define Plastikstufe, 
we first introduce an overtwisted disc in a $3$-dimensional contact manifold. 
Let $(M,\xi)$ be a $3$-dimensional contact manifold, 
and $D\subset(M,\xi)$ an embedded disc. 
 The contact structure $\xi$ trace a singular $1$-dimensional foliation 
$D_\xi$ on $D$ called the \emph{characteristic foliation}. 
 The disc $D$ is called an \emph{overtwisted disc}\/
if the characteristic foliation is isomorphic to Figure~\ref{fig:3otdiscs}(1). 
%
%
\begin{figure}[htb]
  \centering
  \begin{small}
    \includegraphics[height=3.6cm]{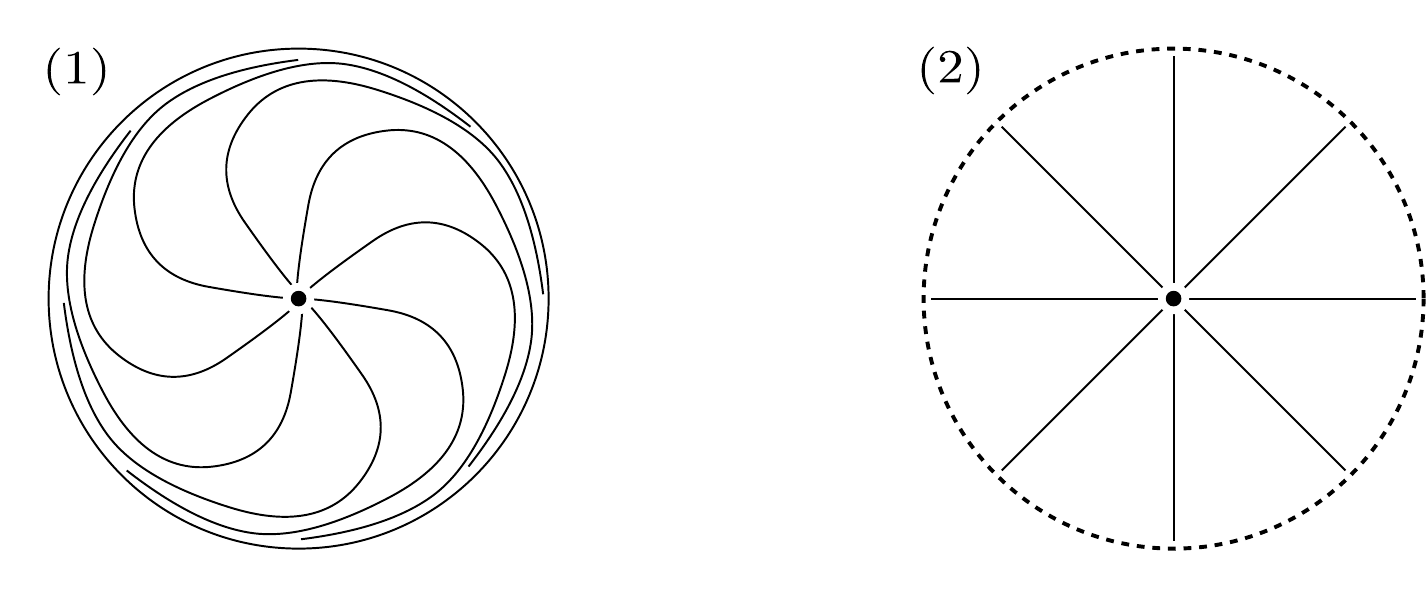}
  \end{small}
  \caption{overtwisted disc in dimension~$3$.}
  \label{fig:3otdiscs}
\end{figure} 
 The boundary $\rd D$ is a Legendrian circle and the center is a singular point.
 It is obtained from a non-generic disc in Figure~\ref{fig:3otdiscs}(2) 
by perturbing slightly. 
 In the second disc, the boundary and the center is the set of singular points. 

  Then the plastikstufe is defined as follows. 
 Let $(M,\xi)$ be a contact manifold of dimension~$2n+1>3$, 
and $B$ a closed manifold of dimension~$n-1$. 
%
%
\begin{dfn}
  A \emph{plastikstufe\/} with \emph{core\/} $B$ is a submanifold 
$\mathcal{P}_B\subset(M,\xi)$ diffeomorphic to $D^2\times B$ 
which satisfies the following conditions: 
\begin{itemize}
\setlength{\parskip}{0cm} 
\setlength{\itemsep}{0cm}
  \item each fiber $\{z\}\times B$ is tangent to $\xi$ for any $z\in D^2$, 
  \item on each slice $D^2\times\{b\}$, 
    $\xi\cap T(D^2\times\{b\})$ generates the same singular foliation 
    as the overtwisted disc (see Figure~\ref{fig:3otdiscs}(1)) for any $b\in B$.
\end{itemize}
 A plastikstufe is said to be \emph{small\/} 
if it is contained in an embedded open ball in $(M,\xi)$. \\
 The submanifold in a plastikstufe $\mathcal{P}_B\subset(M,\xi)$ 
corresponding to $(0,1)\times B\subset D^2\times B$ is a Legendrian submanifold,
where $(0,1)$ is a leaf of the characteristic foliation on the overtwisted disc 
(see Figure~\ref{fig:3otdiscs}). 
 A thin ribbon corresponding to $(0,\epsilon)\times B\subset(0,1)\times B$ 
sufficiently close to the core $B$ is called a \emph{leaf ribbon}\/ 
of the plastikstufe $\mathcal{P}_B$. 
 All leaf ribbons are isotopic as Legendrian submanifolds. \\
 A contact structure $\xi$ is said to be \emph{PS-overtwisted\/} 
if there exists a plastikstufe in $(M,\xi)$. 
\end{dfn} 

  An important property of plastikstufe is the following 
due to Niederkr\"uger~\cite{niederkruger06}: 
%
%
\begin{thrm}
  If a closed contact manifold has a plastikstufe, 
then it can not have any (semi-positive) symplectic filling. 
\end{thrm} 

¡¡  A plastikstufe has a standard neighborhood 
although it has codimension $n>1$. 
¡¡Let $\mathcal{P}_B$ be a plastikstufe with core $B$ 
in a contact manifold $(M,\xi)$ of dimension~$2n+1$. 
 It is proved that there exists the standard tubular neighborhood 
of $\mathcal{P}_B\subset(M,\xi)$ described as follows (see \cite{muniplst}). 
 Setting $\alpha_\ot:=\cos rdz+\sin rd\theta$, 
we have the standard overtwisted contact form $\alpha_\ot$ 
on $\R^3$ with the cylindrical coordinates $(z,r,\theta)$. 
 Let $D^2_\ot\subset(\R^3,\ker\alpha_\ot)$ denote 
an overtwisted disc. 
 Setting $\alpha_\textup{PS}:=\alpha_\ot+\lambda_\textup{can}$, 
where $\lambda_\textup{can}=\mathbf{q}d\mathbf{p}$ 
is the canonical Liouville $1$-form on $T^\ast B$, 
we have a contact form on $\R^3\times T^\ast B$. 
 Let $\xi_\textup{PS}=\ker\alpha_\textup{PS}$ 
denote the corresponding contact structure. 
 Then there exists a tubular neighborhood $U_\textup{PS}\subset(M,\xi)$ 
of $\mathcal{P}_B$ and a contact embedding 
$\vphi_\textup{PS}\colon(U_\textup{PS},\xi)\to(\R^3\times T^\ast B,\xi_\textup{PS})$
that maps $\mathcal{P}_B$ 
to $\vphi_\textup{PS}(\mathcal{P}_B)=D^2_\ot\times B_0$, 
where $B_0$ is the zero-section of $T^\ast B$. 

\subsection{Loose Legendrian submanifold} \label{sec:looselegendrian}
  Loose Legendrian submanifold is a special Legendrian submanifolds 
in a contact manifold of dimension greater than $3$. 
 It is introduced by Murphy~\cite{murphy12} 
as a class where a parametric \textit{h}-principle holds. 
 In this subsection, we review the definition and basic properties. 

  First, we define loose Legendrian submanifold. 
 It is defined by using model chart. 
 We introduce some parts of the model. 
 Let $\xist$ be the standard contact structure on $\R^3$, 
and $\alphast:=dz-ydx$ the standard contact form defining $\xist$. 
 In $(\R^3,\xist)$, let $L_0$ be a negatively stabilized Legendrian curve 
as in Figure~\ref{fig:negstab}(4). 
%
%
\begin{figure}[htb]
  \centering
  \begin{small}
    \includegraphics[height=2.1cm]{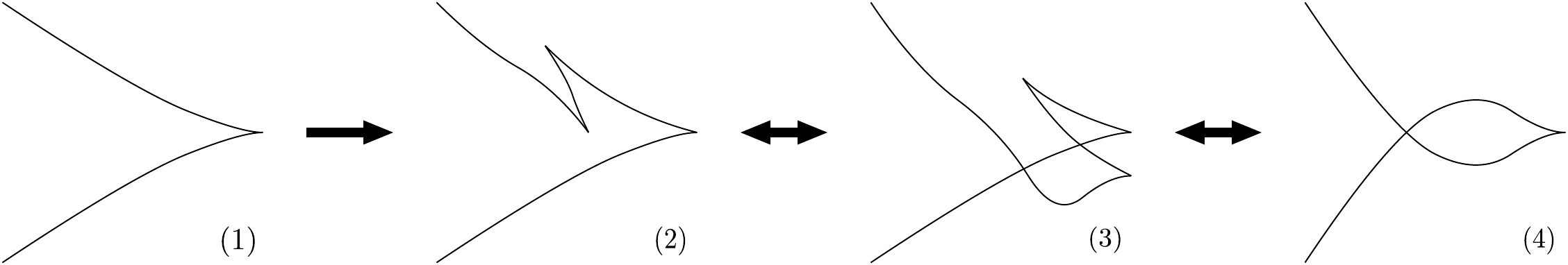}
  \end{small}
  \caption{Stabilization of a Legendrian arc (front projection).}
  \label{fig:negstab}
\end{figure} 
 More precisely, it is a part of the curve 
\begin{equation*}
  \lft(t^2,\frac{15}{4}(t^3-t),\frac{3}{2}t^5-\frac{5}{2}t^3\rgt) 
\end{equation*}
including the cusp and the crossing point in Figure~\ref{fig:negstab}(4). 
 Let $W\subset(\R^3,\xist)$ a convex open ball that contains $L_0$. 
 Next, we introduce some symplectic parts. 
 Let $\lambdast=\sum q_idp_i$ be the standard Liouville form 
on $T^\ast\R^{n-1}$ and $Z\subset T^\ast\R^{n-1}$ the Lagrangian zero section. 
 Set 
\begin{equation*}
  V_\rho:=\lft\{(p_1,\dots,p_{n-1},q_1,\dots,q_{n-1})\in T^\ast\R^{n-1}\mid 
  p_1^2+\dots+p_{n-1}^2<\rho^2,\ q_1^2+\dots+q_{n-1}^2<\rho^2\rgt\}. 
\end{equation*}
Then the $1$-form $\alphast+\lambdast$ is a contact form 
on $\R^3\times T^\ast\R^{n-1}\cong\R^{2n+1}$, 
which is nothing but the standard contact structure. 
 We abuse the notation $\xist$ for $\ker(\alphast+\lambdast)$ as well. 
 The submanifold $L_0\times Z\subset(\R^{2n+1},\xist)$ 
is a Legendrian submanifold. 

  Now, loose Legendrian submanifold is defined as follows. 
 It is defined by Murphy~\cite{murphy12}. 
 The following definition is due to~\cite{muniplst}, 
which is equivalent to the original one. 
%
%
\begin{dfn}
  The relative pair $(W\times V_\rho,L_0\times Z)$ 
of an open set and a Legendrian submanifold in $(\R^{2n+1},\xist)$ 
for some convex open set $W\subset\R^3$ containing $L_0$ 
is called a \emph{loose chart}\/ if $\rho>1$. 

  Let $(M,\xi)$ be a contact manifold of dimension~$2n+1>3$. 
  A connected Legendrian submanifold $\Lambda\subset(M,\xi)$ 
is said to be \emph{loose\/} 
if there exists an open set $U\subset M$ so that $(U,U\cap\Lambda)$ 
is contactomorphic to a loose chart. 
\end{dfn} 

  The notion, loose Legendrian, is introduced as a class 
that satisfies the parametric \textit{h}-principle. 
 The following theorem is proved by Murphy~\cite{murphy12}. 
%
%
\begin{thrm}
  Let $(M,\xi)$ be a contact manifold of dimension~$2n+1>3$. 
 If loose Legendrian submanifolds $\Lambda_0,\ \Lambda_1\subset(M,\xi)$ 
are isotopic as embeddings, 
then they are isotopic as Legendrian embeddings. 
\end{thrm} 

  Some relations between loose Legendrian submanifolds and plastikstufes 
are studied in~\cite{muniplst}. 
 In order to state the result, we need some other notions. 
 We mention the relation in Section~\ref{sec:rotPS}. 
 A key observation for the results, as well as for results in this paper, 
is such a relation in dimension~$3$. 
 A relation between negative stabilization of a Legendrian knot 
and an overtwisted disc 
in a contact $3$-manifold is discussed also in~\cite{muniplst}. 
%
%
\begin{thrm}\label{thrm:slidingOTD}
  Let $(M,\xi)$ be an overtwisted contact $3$-manifold 
with an overtwisted disc $D_\ot^2$. 
 Suppose that $L\subset(M,\xi)$ is a Legendrian knot 
which never intersects with the overtwisted disc $D_\ot^2$. 
 Then the Legendrian knot $\tilde{L}:=L\#\rd D_\ot^2$ 
obtained as a Legendrian connected sum 
of $L$ and the boundary $\rd D_\ot^2$ of the overtwisted disc 
is a negative destabilization of $L$. 
 Further, $L$ and $\tilde{L}$ are isotopic as Legendrian knots in $(M,\xi)$. 
\end{thrm} 
\noindent
 It is a key idea in the discussion in Subsection~\ref{sec:PS2OT} 
for the proof of Theorem~\ref{thma}. 

\subsection{Overtwistedness}\label{sec:overtwisted}
  The notion, overtwistedness of a contact structure, 
implies a class of contact structures 
where a parametric \textit{h}-principle holds. 
 For higher dimensions, 
it is introduced by Borman, Eliashberg, and Murphy~\cite{boelmu}. 
 They defined the overtwisted disc in any dimension, 
and proved that the class of contact structures which have the overtwisted discs
satisfies the \textit{h}-principle (Theorem~\ref{thrm:existh-prnc}). 
 However, as mentioned in Introduction, the definition of the overtwisted disc 
is complicated. 
 One of the purposes of this paper is to find another characterization 
than the overtwisted disc. 
 Then we omit the definition of overtwisted disc in this paper. 
 We assume the existence of the class of contact structures that satisfies 
Theorem~\ref{thrm:existh-prnc}. 
 The key tool to prove the overtwistedness is Proposition~\ref{prop:ls2ot} 
below due to Casals, Murphy, and Presas~\cite{camupr}, 
that gives a sufficient condition for the overtwistedness 
in relation to the loose Legendrian submanifolds.

  The most important properties of the class, overtwisted contact structures, 
is that it satisfies the \textit{h}-principle. 
 It is proved by Borman, Eliashberg, and Murphy~\cite{boelmu}. 
 In terms of \textit{h}-principle, 
the formal counterpart for contact structure is almost contact structure. 
 Let $M$ be a $(2n+1)$-dimensional manifold and $A\subset M$ a subset 
which satisfies that $M\setminus A$ is connected. 
 Let $\xi$ be an almost contact structure on $M$ 
which is a genuine contact structure 
on an open neighborhood $\op A$ of $A\subset M$. 
 Then let $\contot{M}{A}{\xi}$ denote the set of contact structures on $M$ 
which are overtwisted on $M\setminus A$ and coincide with $\xi$ on $\op A$, 
and $\cont MA\xi$ the set of almost contact structures on $M$ 
which coincide with $\xi$ on $\op A$. 
 Further, for an embedding $\phi\colon D_{\ot}\to M\setminus A$, 
we introduce the following subsets of the sets above, 
where $D_\ot^2$ is the overtwisted disc with a germ of contact structure. 
 Let $\contot MA{\xi,\phi}\subset\contot MA\xi$ 
and $\mathfrak{cont}_\ot(M;A,\xi,\phi)\subset\cont MA\xi$ 
denote subsets 
consist of contact and almost contact structures 
for each of which $\phi$ is a contact embedding. 
 Then the following is one of the most important theorems in~\cite{boelmu}. 
%
%
\begin{thrm}[Borman, Eliashberg, Murphy~\cite{boelmu}]\label{thrm:existh-prnc}
  Let $j\colon\contot MA\xi\to\cont MA\xi$ be the inclusion mapping. 
 Then the induced mapping 
\begin{equation*}
  j_\ast\colon\pi_0(\contot MA\xi)\to\pi_0(\cont MA\xi)
\end{equation*}
is isomorphic. 
 Moreover, the restriction 
\begin{equation*}
  j|_{\contot MA{\xi,\phi}}\colon\contot MA{\xi,\phi}\to
  \mathfrak{cont}_\ot(M;A,\xi,\phi)
\end{equation*}
is weak homotopy equivalent. 
\end{thrm} 
 As a corollary, the following is proved for isocontact embedding. 
%
%
\begin{cor}[Borman, Eliashberg, Murphy~\cite{boelmu}]\label{cor:isoctemb}
  Let $(M,\xi)$ be a connected overtwisted contact manifold of dimension~$2n+1$ 
and $(N,\zeta)$ be an open contact manifold of the same dimension. 
 Let $f\colon N\to M$ be an embedding covered by a bundle homomorphism 
$\Phi\colon TN\to TM$ which preserves contact structures fiberwise and 
conformal symplectic structure on contact hyperplanes. 
 If $df\colon TN\to TM$ is homotopic to $\Phi$ as bundle monomorphisms, 
then there exists a isocontact embedding $\tilde{f}\colon (N,\zeta)\to(M,\xi)$ 
isotopic to $f$. 
 In particular, 
  if a contact manifold is overtwisted, 
any contact open ball of the same dimension can be embedded into it. 
\end{cor} 
\noindent
 From this corollary, it follows that an overtwisted contact manifold 
should have a plastikstufe (see~\cite{boelmu}). 
 This essentially implies the necessary condition of Theorem~\ref{thma}. 
 We discuss more precisely in Subsection~\ref{sec:OT2PS}. 

  We should mention a relation between loose Legendrian submanifold 
and overtwisted contact structure. 
 To do that, we define trivial Legendrian sphere. 
 Let $\eta_0$ be the standard contact structure 
\begin{equation*}
  \ker\lft.\lft(\sum_{i=1}^{n+1}x_idy_i-y_idx_i\rgt)\rgt|_{TS^{2n+1}} 
\end{equation*}
on the unit sphere $S^{2n+1}\subset\R^{2n+2}$, 
where $(x_1,\dots,x_{n+1},y_1,\dots,y_{n+1})\in\R^{2n+2}$ are coordinates. 
 The $n$-dimensional sphere 
\begin{equation*}
  \Lambda_0:=\lft\{(x_1,\dots,x_{n+1},y_1,\dots,y_{n+1})\in S^{2n+1}\mid 
  y_1=\dots=y_{n+1}=0\rgt\}\subset(S^{2n+1},\eta_0)
\end{equation*}
is Legendrian. 
 Since $(S^{2n+1},\eta_0)\setminus\{point\}$ is contactomorphic 
to the standard contact space $(\R^{2n+1},\xist)$, 
the Legendrian submanifold $\Lambda_0$ is identified 
with a topologically trivial Legendrian sphere in $(\R^{2n+1},\xist)$. 
 Then a Legendrian sphere $\Lambda$ in a contact manifold $(M,\xi)$ 
of dimension~$2n+1$ is said to be \emph{trivial\/} 
if there exists a contact embedding of an open ball 
$U\subset (\R^{2n+1},\eta_0)$ containing $\Lambda$ to $(M,\xi)$ 
which maps $\Lambda_0$ to $\Lambda$. 
 Using these notions, the following relation 
between loose Legendrian submanifold and overtwistedness is proved. 
%
%
\begin{prop}[Casals, Murphy, and Presas~\cite{camupr}]\label{prop:ls2ot}
  Let $(M,\xi)$ be a contact manifold of dimension~$2n+1>3$. 
 If the trivial Legendrian sphere $\Lambda_0\subset(M,\xi)$ is loose, 
then $\xi$ is overtwisted. 
\end{prop} 
\noindent
 This property is used to show the sufficient condition of Theorem~\ref{thma} 
in Subsection~\ref{sec:PS2OT}. 

\subsection{\textit{h}-principle}\label{sec:h-princ}
  In this section, we review the Smale-Hirsch immersion Theorem, 
and Gromov's \textit{h}-principle for subcritical isotropic submanifolds. 
 For the subjects in this subsection, 
the readers should consult~\cite{elmi} and~\cite{madachi} 
as well as~\cite{gromov}. 

  First, we introduce the Smale-Hirsch immersion theorem. 
 Let $M$ and $V$ be manifolds of dimension $n$ and $p$, respectively. 
 Let $\mathrm{Imm}(M,V)$ denote a set of all immersions from $M$ to $V$ 
endowed with $C^\infty$-topology, 
and $\mathrm{Mono}(TM,TV)$ a set of all monomorphisms, 
that is a bundle homomorphisms which are fiberwise injective, from $TM$ to $TV$ 
endowed with compact-open topology. 
%
%
\begin{thrm}[the Smale-Hirsch immersion theorem]\label{thrm:shimmthm}
  If $n=\dim M<p=\dim V$, then the inclusion
\begin{equation*}
  \mathrm{Imm}(M,V)\hookrightarrow\mathrm{Mono}(TM,TV),\qquad 
  f\mapsto df
\end{equation*}
is a weak homotopy equivalence. 
\end{thrm} 
\noindent
 In other words, the parametric $C^0$-dense \textit{h}-principle holds 
for such immersions. 
 The existence of a path of monomorphisms 
implies the existence of a path of immersions. 

  Next, we introduce Gromov's \textit{h}-principle 
for subcritical isotropic embeddings. 
 In this paper, we need a version for contact structures. 
 Let $(W,\xi)$ be a contact manifold of dimension~$2n+1$, 
and $V$ a manifold of dimension~$m<n$. 
 Let $\mathrm{Emb_{isot}}(V,W)$ denote a set of all isotropic embeddings 
from $V$ to $(W,\xi)$, 
and $\mathrm{Mono^{emb}_{isot}}(TV,TW)$ a set of all isotropic monomorphisms, 
that is a bundle homomorphisms which are fiberwise isotropic injective, 
from $TV$ to $(TW,\xi)$. 
%
%
\begin{thrm}[Gromov]\label{thrm:g_h-pr}
  The inclusion 
\begin{equation*}
  \mathrm{Emb_{isot}}(V,W)\hookrightarrow\mathrm{Mono^{emb}_{isot}}(TV,TW), \qquad
  f\mapsto df
\end{equation*}
is a weak homotopy equivalence. 
\end{thrm} 
\noindent
 In other words, the parametric $C^0$-dense \textit{h}-principle holds 
for such embeddings. 
 The existence of a formal isotopy of subcritical isotropic embeddings 
implies the existence of a path of isotropic embeddings. 

\section{Rotation class of plastikstufe}\label{sec:rotPS}
  Rotation class of a plastikstufe is introduced in Subsection~\ref{sec:rotdef}.
 In order to discuss plastikstufe with toric core, 
trivial rotation for such plastikstufe is introduced 
in Subsection~\ref{sec:torictrivial}. 

\subsection{Rotation class and loose Legendrian submanifold}
\label{sec:rotdef}
  We introduce the rotation class of a plastikstufe. 
 Then we introduce some results 
on the relation between plastikstufes with spherical core 
and loose Legendrian submanifolds. 

  The rotation class of a plastikstufe is defined 
as the relative rotation class of the leaf ribbon of the plastikstufe 
with respect to the ``standard'' Legendrian ribbon. 
 It is defined for the small plastikstufes with spherical core 
by Murphy, Niederkr\"uger, Plamenevskaya, and Stipsicz in~\cite{muniplst}. 
 Then we first introduce the relative rotation class 
of two Legendrian immersions. 
 Then we define rotation class of plastikstufes with spherical core 
following~\cite{muniplst}. 
 In Subsection~\ref{sec:torictrivial}, we introduce such notion 
for plastikstufes with ``toric'' core. 

  We define the relative rotation class of Legendrian immersions. 
 Let $(M,\xi)$ be a contact manifold of dimension~$2n+1$ 
and $\alpha$ a contact form defining $\xi$. 
 And let $f,g\colon\Lambda\to(M,\xi)$ be two Legendrian immersions. 
 Assume that there exists an open ball $U\subset M$ 
that contains both of the image $f(\Lambda)$, $g(\Lambda)$. 
 First, we regard the contact structure $\xi|_U$ restricted to $U$ 
as the trivial complex vector bundle as follows. 
 Let $J$ be an almost complex structure on $\xi$ 
compatible with the conformal symplectic structure induced by $d\alpha|_\xi$. 
 By taking a $J$-complex trivialization of $(\xi,J)$, 
we can regard $(\xi|_U,J)$ as the trivial bundle $\C^n\times U\to U$. 
 Then, by this identification, $df$ can be regarded as the bundle map 
$df\colon T\Lambda\to\Lambda\times\C^n$, 
for which $(df)_x(T_x\Lambda)\subset(\xi_{f(x)},J)$ is totally real 
since $f\colon\Lambda\to(M,\xi)$ is Legendrian. 
 Therefore, we have the complexification 
\begin{equation*}
  df^\C\colon T\Lambda\otimes\C\to\Lambda\times\C^n
\end{equation*}
which is fiberwise complex isomorphism. 
 For another Legendrian immersion $g\colon\Lambda\to(M,\xi)$, 
we also have a complex bundle map 
$dg^\C\colon T\Lambda\otimes\C\to\Lambda\times\C^n$ 
which is fiberwise complex isomorphism. 
 Then, from the two maps $df^\C$ and $dg^\C$, 
a mapping $\varphi\colon\Lambda\to\textup{GL}(n,\C)$ is determined by 
\begin{equation*}
  (dg^\C)_x=\varphi(x)\circ(df^\C)_x,\qquad x\in\Lambda. 
\end{equation*}
 Now, we define the \emph{relative rotation class}\/ of $g$ with respect to $f$ 
as the homotopy class of $\varphi$ 
in $[\Lambda,\textup{GL}(n,\C)]\cong[\Lambda,U(n)]$. 

  The rotation class of a plastikstufe is defined as follows. 
 Let $\mathcal{P}_{S^{n-1}}$ be a small plastikstufe with spherical core $S^{n-1}$
in a contact manifold $(M,\xi)$ of dimension~$2n+1$. 
 Then a leaf ribbon of $\mathcal{P}_{S^{n-1}}$ is a Legendrian submanifold 
of $(M,\xi)$ diffeomorphic to $(0,1)\times S^{n-1}$. 
 Let $f\colon (0,1)\times S^{n-1}\to(M,\xi)$ 
be the Legendrian embedding corresponding to the leaf ribbon. 
 Since $\mathcal{P}_{S^{n-1}}$ is small, we can apply the discussion above. 
 On the other hand, let $D\subset(M,\xi)$ be a Legendrian disc. 
 Note that all Legendrian discs are isotopic as Legendrian submanifolds. 
 Then we take $\inter D\setminus\{\text{point}\}\subset(M,\xi)$ 
as the standard Legendrian submanifold diffeomorphic to $(0,1)\times S^{n-1}$. 
 Let $g\colon(0,1)\times S^{n-1}\to(M,\xi)$ 
be the Legendrian embedding 
corresponding to the standard Legendrian $(0,1)\times S^{n-1}$. 
 Now, the \emph{rotation class}\/ of the small plastikstufe 
$\mathcal{P}_{S^{n-1}}$ 
is defined as the relative rotation class of $f$ with respect to $g$. 
 The small plastikstufe $\mathcal{P}_{S^{n-1}}$ is said 
to have \emph{trivial rotation}\/ if the rotation class vanishes. 

  Using this notion, rotation class, a relation between plastikstufe 
and loose Legendrian submanifold is proved. 
%
%
\begin{prop}[Murphy, Niederkr\"uger, Plamenevskaya, Stipsicz~\cite{muniplst}]
  Let $\mathcal{P}_{S^{n-1}}$ be a small plastikstufe with spherical core 
and trivial rotation in a contact manifold $(M,\xi)$ of dimension~$2n+1>3$. 
 Then any Legendrian submanifold $\Lambda\subset(M,\xi)$ 
disjoint from $\mathcal{P}_{S^{n-1}}$ is loose. 
\end{prop} 
\noindent
  Combining this with Proposition~\ref{prop:ls2ot}, 
the following relation between plastikstufes and overtwistedness follows 
as a corollary. 
%
%
\begin{cor}[Casals, Murphy, Presas~\cite{camupr}]
  Let $(M,\xi)$ be a contact manifold of dimension~$2n+1>3$. 
 If there exists a small plastikstufe $\mathcal{P}_{S^{n-1}}\subset(M,\xi)$ 
with spherical core and trivial rotation 
then the contact structure $\xi$ is overtwisted. 
\end{cor} 

\subsection{Trivial rotation of a plastikstufe with toric core}
\label{sec:torictrivial}
  The rotation class of a plastikstufe with toric core is defined 
in this subsection. 
 Recall that the rotation class of a plastikstufe with spherical core 
is defined, in the previous subsection, 
as a relative rotation class with respect to the punctured Legendrian disc. 
 Instead of that Legendrian submanifold, 
we need the ``standard'' Legendrian submanifold 
for plastikstufes with toric core. 
 We first define such Legendrian submanifold, then define the rotation class. 

  First of all, we observe what the ``standard'' Legendrian submanifold 
should be. 
 Let $(M,\xi)$ be a contact manifold of dimension~$2n+1$. 
 Recall that the ``standard'' Legendrian submanifold 
for plastikstufes with spherical core $S^{n-1}$ is the punctured Legendrian disc 
$\inter D^n\setminus\{\text{point}\}$ (see Subsection~\ref{sec:rotdef}). 
 In other words, it is a Legendrian submanifold diffeomorphic 
to the leaf ribbon $(0,1)\times S^{n-1}$ which is unique up to isotopy 
(see Subsection~\ref{sec:plastikstufe} for definition of leaf ribbon). 
 Therefore, for plastikstufe with toric core $T^{n-1}$, 
we need a Legendrian submanifold 
diffeomorphic to the leaf ribbon $(0,1)\times T^{n-1}$ 
which is unique up to isotopy. 

  We define the ``standard'' Legendrian submanifold 
diffeomorphic to $(0,1)\times T^{n-1}$ 
in a contact manifold $(M,\xi)$ of dimension~$2n+1>3$. 
 If $n=2$ then $T^{n-1}=T^1=S^1$. 
 Then the definition is the same as the spherical case. 
 We assume $n>2$. 
 Let $D^n\subset(M,\xi)$ be a Legendrian disc 
and $T^{n-2}\subset D^n$ a trivially embedded $(n-2)$-dimensional torus. 
 In other words, $T^{n-2}\subset D^n$ is unknotted 
(i.e.\ the boundary of a handlebody). 
 Note that the choice of $T^{n-2}\subset(M,\xi)$ is unique up to isotopy 
as isotropic submanifolds. 
 Let $U\subset D^n$ be a tubular neighborhood of $T^{n-2}$, 
which is diffeomorphic to $T^{n-2}\times D^2$. 
 According to this identification, set 
\begin{equation*}
  \tilde{U}:=T^{n-2}\times\lft(\inter D^2\setminus\{0\}\rgt)
  \subset D^n\subset(M,\xi).
\end{equation*}
 Then $\tilde{U}\subset(M,\xi)$ is a Legendrian submanifold 
diffeomorphic to $(0,1)\times T^{n-1}$ unique up to isotopy. 

  Now, we define rotation class for plastikstufe with toric core. 
 Let $\mathcal{P}_{T^{n-1}}$ be a small plastikstufe with toric core 
in a contact manifold $(M,\xi)$ of dimension~$2n+1>3$. 
%
%
\begin{df}\label{df:rtcl}
  The \emph{rotation class}\/ of the small plastikstufe $\mathcal{P}_{T^{n-1}}$ 
is the relative rotation class 
of the Legendrian embedding $f\colon(0,1)\times T^{n-1}\to(M,\xi)$ 
for a leaf ribbon of $\mathcal{P}_{T^{n-1}}$ 
with respect to the Legendrian embedding $g\colon(0,1)\times T^{n-1}\to(M,\xi)$ 
for the standard Legendrian submanifold $\tilde{U}$. 
 The small plastikstufe $\mathcal{P}_{T^{n-1}}$ is said 
to have \emph{trivial rotation}\/ if the rotation class vanishes. 
\end{df}

\section{Proof of Theorem~\ref{thma}}\label{sec:proof}
  Theorem~\ref{thma} is proved in this section. 
 The ideas of the proof are debt to~\cite{muniplst}, \cite{boelmu}, 
and~\cite{camupr}. 
 We show the necessity (``if'' part) in Subsection~\ref{sec:OT2PS}, 
and the sufficiency (``only if'' part) in Subsection~\ref{sec:PS2OT}. 

\subsection{Overtwistedness implies the existence of a plastikstufe}
\label{sec:OT2PS}
  In this subsection, assuming the overtwistedness, 
we show the existence of a small plastikstufe 
with toric core whose rotation is trivial. 
 In other words, we show the following. 
%
%
\begin{prop}\label{prop:necess}
  Let $(M,\xi)$ be a contact manifold of dimension~$2n+1>3$. 
 If the contact structure $\xi$ is overtwisted, 
then there exists a small plastikstufe with toric core 
that has trivial rotation. 
\end{prop} 

  The existence of a plastikstufe is discussed in~\cite{boelmu}
as an existence of a contact embedding of the model plastikstufe, 
which is defined as follows. 
 In other words, it is the same as the standard tubular neighborhood 
of a plastikstufe (see Subsection~\ref{sec:plastikstufe}). 
 Let $B$ be a closed manifold of dimension~$n-1$. 
 Then we have a contact manifold 
$\lft(\R^3\times T^\ast B, \ker(\alpha_{\ot}+\lambda_{T^\ast B})\rgt)$ 
of dimension~$2n+1$, 
where $\alpha_\ot=\cos rdz+\sin rd\theta$ 
is the standard overtwisted contact form on $\R^3$ 
with the cylindrical coordinates $(r,\theta,z)$,
and $\lambda_{T^\ast B}=\sum p_idq_i$ 
is the canonical Liouville form on $T^\ast B$ with coordinates $(q_i,p_i)$. 
 In this contact manifold, the submanifold $D^2_\ot\times B_0$ 
is a plastikstufe with core $B$, 
where $D^2_\ot\subset(\R^3,\ker\alpha_{\ot})$
is a simple overtwisted disc, 
and $B_0\cong B\subset T^\ast B$ is the zero-section. 
 Then let $(\mathcal{P}_B,\zeta)$ denote the pair of germs 
of $(2n+1)$-dimensional manifold and contact structure 
along the plastikstufe $D^2_\ot\times B_0\subset\R^3\times T^\ast B$.  
 It is called the \emph{model plastikstufe with core}\/ $B$. 
 The following is obtained as a corollary of Theorem~\ref{thrm:existh-prnc} 
via Corollary~\ref{cor:isoctemb}. 
%
%
\begin{cor}[Borman, Eliashberg, Murphy~\cite{boelmu}]\label{cor:possiblePS}
  Let $(M,\xi)$ be an overtwisted contact manifold of dimension~$2n+1>3$, 
and $(\mathcal{P}_B,\zeta)$ the model plastikstufe with core $B$. 
 If the complexification $TB\otimes\C$ of the tangent bundle of $B$ is trivial,
then there exists a contact embedding of $(\mathcal{P}_B,\zeta)$ 
into $(M,\xi)$. 
\end{cor} 

  Now, assuming the overtwistedness, we show the existence 
of a small plastikstufe with toric core that has trivial rotation. 

\begin{proof}[Proof of Proposition~\ref{prop:necess}]
  Let $(M,\xi)$ be an overtwisted contact manifold of dimension~$2n+1>3$, 
and $(\mathcal{P}_{T^{n-1}},\zeta)$ the model plastikstufe 
with toric core $T^{n-1}$. 
 As the tangent bundle $T(T^{n-1})$ is trivial, 
its complexification $T(T^{n-1})\otimes\C$ is also trivial.
 Then, by Corollary~\ref{cor:possiblePS}, 
the model plastikstufe $(\mathcal{P}_{T^{n-1}},\zeta)$ can be embedded 
into $(M,\xi)$. 
 In the proof of Corollary~\ref{cor:possiblePS} in~\cite{boelmu}, 
the embedding is constructed by $h$-principle using the Darboux chart.
 Therefore, $(\mathcal{P}_{T^{n-1}},\zeta)$ is embedded into an open ball 
in $(M,\xi)$. 
 Then the plastikstufe is small. 

  It remains to show that the embedded plastikstufe has trivial rotation. 
 In order to discuss the rotation class of the embedded plastikstufe, 
we should observe the embedded image of a leaf ribbon of the model plastikstufe.
 Then, we should review the construction of the contact embedding 
of the model plastikstufe 
in the proof of Corollary~\ref{cor:possiblePS} in~\cite{boelmu}. 
 The contact embedding of the model plastikstufe 
$(\mathcal{P}_{T^{n-1}},\zeta)$ into $(M,\xi)$ is constructed 
by \textit{h}-principle from two contact bundle-mappings: 
\begin{align*}
  \Psi\colon&\lft(T(\R^3\times T^\ast T^{n-1}), 
              \ker(\alpha_\ot+\lambda_{T^\ast T^{n-1}})\rgt)\to
              \lft(T(\R^3\times T^\ast T^{n-1}), 
              \ker(\alpha_\textup{st}+\lambda_{T^\ast T^{n-1}})\rgt), \\
  \Phi\colon&\lft(T(\R^3\times T^\ast T^{n-1}), 
              \ker(\alpha_\textup{st}+\lambda_{T^\ast T^{n-1}})\rgt)\to
              \lft(T\R^{2n+1},\xi_\textup{st}\rgt), 
\end{align*}
where $\alpha_\textup{st}=d\phi+\sum r^2d\theta$ 
is the standard contact form on $\R^3$, 
and $\xi_\textup{st}=\ker\lft\{d\phi+\sum r_i^2d\theta_i\rgt\}$ 
is the standard contact structure on $\R^{2n+1}$. 
 Note that 
the first bundle-mapping is on the identity of $\R^3\times T^\ast T^{n-1}$, 
and that the second one is on the mapping 
$\R^3\times T^\ast T^{n-1}=\R^3\times(T^\ast S^1)^{n-1}
\to\R^3\times\C^{n-1}=\R^{2n+1}$. 
 In general the second one is constructed, 
under the condition that $TQ\otimes\C$ is trivial, by using Gromov's 
\textit{h}-principle. 
 In this case, it is constructed from the explicit mapping 
from $T^\ast S^1$ to $\C\setminus\{0\}$. 
 From these $\Phi$, $\Psi$, and and the Darboux chart, 
we have a contact bundle homomorphism 
$\lft(T(\R^3\times T^\ast T^{n-1}), \ker(\alpha_\ot+\lambda_{T^\ast T^{n-1}})\rgt)
\to(TM,\xi)$. 
 Then, by Corollary~\ref{cor:isoctemb}, we have a contact embedding 
of $\lft(\R^3\times T^\ast T^{n-1},\ker\{\alpha_\ot+\lambda_{T^\ast T^{n-1}}\}\rgt)$ 
into $(M,\xi)$ that is isotopic to the mapping between the bases 
of the bundle homomorphisms above. 
 The contact embedding of the model plastikstufe is obtained 
as a restriction of the mapping 
to the model plastikstufe 
$\mathcal{P}_{T^{n-1}}=D_\ot^2\times T_0^{n-2}
\subset\R^3\times T^\ast T^{n-1}$. 

  By observing the contact embedding above, 
it is proved as follows that the embedded model plastikstufe 
has trivial rotation. 
 We show that an embedded leaf ribbon is the standard Legendrian submanifold 
diffeomorphic to $(0,1)\times T^{n-1}$ 
in the sense of Subsection~\ref{sec:torictrivial}. 
 The model plastikstufe $\mps=D_\ot^2\times T_0^{n-1}$ 
in $\lft(\R^3\times T^\ast T^{n-1},\ker\{\alpha_\ot+\lambda_{T^\ast T^{n-1}}\}\rgt)$
can be regarded as 
\begin{equation*}
\begin{array}{ll}
  \quad\mps&  \\
  =D_\ot^2\times S_0^1\times T_0^{n-2}
           &\subset\lft(\R^3\times T^\ast S^1\times T^\ast T^{n-2}, 
             \ker\{\alpha_\ot+\lambda_{T^\ast S^1}
             +\lambda_{T^\ast T^{n-2}}\}\rgt), \\
  =D_\ot^2\times S_0^1\times(S_0^1\times\dots\times S_0^1)
           &\subset\lft(\R^3\times T^\ast S^1\times\dots\times T^\ast S^1, 
             \ker\{\alpha_\ot+\lambda_{T^\ast S^1}+\dots
             +\lambda_{T^\ast S^1}\}\rgt). 
\end{array}
\end{equation*}
where $S_0^1\subset T^\ast S^1$ and $T_0^{n-2}\subset T^\ast T^{n-2}$ 
are zero-sections. 
 Then a leaf ribbon is 
$(0,1)\times S_0^1\times T_0^{n-2}\subset \mps$, 
where $(0,\epsilon)\subset D_\ot^2$ is an open segment 
on a leaf of the characteristic foliation on $D_\ot^2$ 
(see Figure~\ref{fig:3otdiscs}). 
 We should recall that each factor $T^\ast S^1$ is mapped 
to $\C\setminus\{0\}\subset\C$ by the mapping $\Phi$ above. 
 Then the image of $T^{n-2}$ is unknotted. 
 Similarly, the image of $(0,1)\times S_0^1$ is isotopic to 
$\inter D^2\setminus\{0\}\subset\C\setminus\{0\}\subset\C$. 
 Then we conclude, after applying the \textit{h}-principle, 
that the embedded leaf ribbon is the standard Legendrian submanifold 
in $(M,\xi)$. 
 In other words, the embedded model plastikstufe has trivial rotation. 
\end{proof}

\subsection{The existence of a plastikstufe implies overtwistedness}
\label{sec:PS2OT}
  In this subsection, assuming the existence of a small plastikstufe 
with toric core and trivial rotation, we show the overtwistedness. 
 The proof is largely debt to the relation between Loose Legendrian submanifold 
and overtwistedness, Proposition~\ref{prop:ls2ot} due to~\cite{camupr}. 
 The main contribution of this paper is the relation 
between plastikstufe with toric core and loose Legendrian submanifolds, 
Theorem~\ref{thrm:pstc2loose} bellow. 
 Combining these results, we obtain the main issue Proposition~\ref{prop:suff} 
of this subsection in Subsubsection~\ref{sec:pfps2pt}. 

\subsubsection{Plastikstufe to loose Legendrian}\label{sec:ps2ll}
  We prove the following claim on a relation 
between plastikstufes with toric core and loose Legendrian submanifolds. 
%
%
\begin{thrm}\label{thrm:pstc2loose}
   Let $(M,\xi)$ be a contact manifold of dimension~$2n+1>3$. 
 Suppose that there exists a small plastikstufe $\mathcal{P}\subset(M,\xi)$ 
with toric core and trivial rotation. 
 Then any Legendrian submanifold which is disjoint from $\mathcal{P}$ is loose. 
\end{thrm} 

  In order to prove this theorem, we need the following lemma. 
 It is a key to the proof of Theorem~\ref{thrm:pstc2loose}. 
 For the statement, we recall that a plastikstufe ${\cal P}$ 
in a contact manifold $(M,\xi)$ has the standard tubular neighborhood 
$\vphi_{\textup{PS}}\colon(U_{\textup{PS}},\xi)
\to\lft(\R^3\times T^\ast T^{n-1},\xi_{\textup{PS}}\rgt)$ 
for which $\vphi_{\textup{PS}}({\cal P})=D_\ot^2\times T_0^{n-1}$, 
where $T_0^{n-1}\subset T^\ast T^{n-1}$ is the zero section 
(see Subsection~\ref{sec:plastikstufe}). 
%
%
\begin{lem}\label{lem:key}
  Let $(M,\xi)$ be a contact manifold of dimension~$2n+1$, 
and $\mathcal{P}\subset(M,\xi)$ 
a small plastikstufe with toric core that has trivial rotation. 
 Then, for any Legendrian submanifold $\Lambda\subset(M,\xi)$ 
disjoint from the plastikstufe $\mathcal{P}$, 
there exist a Legendrian submanifold $\Lambda_0\subset\Lambda$ 
diffeomorphic to $(0,1)\times T^{n-1}$ 
and an ambient contact isotopy 
\begin{equation*}
  \varphi\colon(M,\xi)\times[0,1]\to(M,\xi)
\end{equation*}
that satisfy the following conditions\textup{:} 
setting $\varphi_t(\ \cdot\ ):=\varphi(\ \cdot\ ,t)$, 
\begin{itemize} \setlength{\parskip}{0cm} \setlength{\itemsep}{0cm}
\item $\varphi_t=\textup{id}$ near $\mathcal{P}$ for any $t\in[0,1]$, 
\item $\varphi_0=\textup{id}$, 
\item $\varphi_1(\Lambda_0)$ lies in a standard tubular neighborhood 
  $U_{\textup{PS}}\subset(M,\xi)$ of the plastikstufe $\mathcal{P}$, and 
  is diffeomorphic to $(0,1)\times T_0^{n-1}$ by 
  $\vphi_{\textup{PS}}\colon(U_{\textup{PS}},\xi)
  \to\lft(\R^3\times T^\ast T^{n-1},\xi_\textup{PS}\rgt)$, 
  where $(0,1)\subset(\R^3,\xi_\ot)$ is a Legendrian segment 
  on an extended overtwisted disc, 
  and $T_0^{n-1}\subset T^\ast T^{n-1}$ is the zero section. 
\end{itemize}
\end{lem} 
\begin{proof}
  First of all, we arrange the situation for this proof. 
 A candidate of $\Lambda_0$ is given as follows. 
 Since the plastikstufe $\mathcal{P}\subset(M,\xi)$ is small, 
there exists an open ball $U\subset(M,\xi)$ that includes $\mathcal{P}$. 
 We may assume that $U$ intersects with the given Legendrian submanifold 
$\Lambda\subset(M,\xi)$. 
 Taking a Legendrian disc $D^n_L\subset\Lambda\cap U\subset(M,\xi)$, 
we have, in $D_L^n$, the standard Legendrian submanifold 
diffeomorphic to $(0,1)\times T^{n-1}$
by the construction in Subsection~\ref{sec:torictrivial}. 
 Let $\Lambda_0$ denote the Legendrian submanifold, 
and $f_0\colon (0,1)\times T^{n-1}\to(M,\xi)$ 
be the corresponding Legendrian embedding: $\operatorname{Im}f_0=\Lambda_0$. 
 On the other hand, fix a Legendrian strip $\Lambda_1\subset(M,\xi)$, 
diffeomorphic to $(0,1)\times T^{n-1}$, 
which is isotopic to a leaf ribbon of $\mathcal{P}$ 
but disjoint from $\mathcal{P}$. 
 Let $f_1\colon (0,1)\times T^{n-1}\to(M,\xi)$ 
be a corresponding Legendrian embedding. 
 We are going to connect these two Legendrian embeddings $f_0$, $f_1$ 
by a path of Legendrian embeddings. 
 And then it is extended to an ambient contact isotopy. 

  The proof is divided into the following three steps. 
 We first concentrate on the core-tori of Legendrian submanifolds 
diffeomorphic to $(0,1)\times T^{n-1}$, 
that are subcritical isotropic submanifolds. 
 In Step~1, we connect the restrictions of $f_0$ and $f_1$ to the core-tori 
by a path of subcritical isotropic embeddings. 
 Then, in Step~2, we extend it to a path of Legendrian embeddings. 
 And then it is extended to an ambient contact isotopy in Step~3. 

%
\noindent\underline{\textbf{Step~1:}}\ 
  We will find the path of subcritical isotropic embeddings 
using Gromov's \textit{h}-principle for subcritical isotropic embeddings 
(Theorem~\ref{thrm:g_h-pr}). 
 In order to apply Gromov's \textit{h}-principle, 
we should construct a formal isotopy. 
 Further, the needed formal isotopy is obtained 
by using The Smale-Hirsch immersion theorem (Theorem~\ref{thrm:shimmthm}). 

¡¡What we should do first is to construct a formal monomorphism 
for Theorem~\ref{thrm:shimmthm}. 
 Setting $f_i^{\textup{cr}}:=f_i|_{T^{n-1}\times\{c\}}$, $i=0,1$, for some $c\in(0,1)$,
we have two subcritical isotropic embeddings of a torus $T^{n-1}$: 
\begin{equation*}
  f_i^{\textup{cr}}\colon T^{n-1}\to (M,\xi). 
\end{equation*}
 We extend these embeddings to a certain embedding 
$T^{n-1}\times[0,1]\to(M,\xi)$ by Theorem~\ref{thrm:shimmthm}. 
 To apply it, what we need is a formal monomorphism 
$T(T^{n-1}\times[0,1])\to TM$. 
 We construct a Lagrangian monomorphism as a real part of a complexification 
as follows. 

  First, we extend from the both ends a little. 
 As restrictions of $f_0$ and $f_1$, we have two Legendrian embeddings 
\begin{equation*}
  f_0\colon T^{n-1}\times[0,\delta]\to(M,\xi), \qquad
  f_1\colon T^{n-1}\times[1-\delta,1]\to(M,\xi),
\end{equation*}
reparameterizing $(0,1)$ so that $f_i|_{T^{n-1}\times\{i\}}=f_i^{\textup{cr}}$, 
$i=0,1$, for some small $\delta>0$. 
 Taking complexifications of tangent bundles 
and complex trivialization of $\xi$ on $U$ 
(see Subsection~\ref{sec:rotdef}), 
we have the following bundle mappings 
\begin{equation}
  df_0^\C\colon T(T^{n-1}\times[0,\delta])\otimes\C\to\C^n, \qquad 
  df_1^\C\colon T(T^{n-1}\times[1-\delta,1])\otimes\C\to\C^n, 
\end{equation}
which are fiberwise complex isomorphic. 
 We remark that targets $\C^n$ of the mappings above 
are fibers of $U\times\C^n\to U$, the trivialization of $\xi|_U$. 
 The mapping $df_0^\C$ is homotopic to the bundle mapping 
$G_0^\C\colon T(T^{n-1}\times[0,\delta])\otimes\C\to\C^n$ defined by 
\begin{equation*}
  (G_0^\C)_{(x,t)}:=(df_0^\C)_{(x,\delta)}\colon (T_x(T^{n-1})\times\R)\otimes\C
  \to\C^n,\qquad 0\le t\le\delta. 
\end{equation*}
 Similarly,  The mapping $df_1^\C$ is homotopic to the bundle mapping 
$G_1^\C\colon T(T^{n-1}\times[1-\delta,1])\otimes\C\to\C^n$ defined by 
\begin{equation*}
  (G_1^\C)_{(x,t)}:=(df_1^\C)_{(x,1-\delta)}\colon (T_x(T^{n-1})\times\R)\otimes\C
  \to\C^n,\qquad 1-\delta\le t\le1. 
\end{equation*} 

  Next, we extend the mappings $G_i^\C$, $i=0,1$, above 
to the whole $T(T^{n-1}\times[0,1])\otimes\C$. 
 These $G_i^\C$, $i=0,1$, are independent of the choices 
of $t\in[0,\delta]$ or $t\in[1-\delta,1]$ respectively, 
and are also fiberwise complex isomorphic.  
 Then, for a mapping $\psi\colon T^{n-1}\to\textup{GL}(n,\C)$ 
satisfying $(df_0^C)_{(x,\delta)}=\psi(x){\cdot}(df_1^\C)_{(x,1-\delta)}$, 
we have 
\begin{equation*}
  (G_0^\C)_{(x,t)}=\psi(x){\cdot}(G_1^\C)_{(x,t)}. 
\end{equation*}
 On the other hand, the plastikstufe $\mathcal{P}$ has trivial rotation, 
and a Legendrian strip $\Lambda_1'=f_1(T^{n-1}\times(0,\delta))$ is isotopic 
to a leaf-ribbon of $\mathcal{P}$. 
 In addition, $\Lambda_0'=f_0(T^{n-1}\times(1-\delta,1))$ 
is the standard Legendrian submanifold. 
 Then the mapping $\psi$ is homotopic to the constant mapping 
$e\colon T^{n-1}\to\{e\}\subset\textup{GL}(n,\C)$ 
to the identity $e\in\textup{GL}(n,\C)$. 
 By this homotopy, we have a homotopy $df_t^\C$ between $df_0^\C$ and $df_1^C$ 
through fiberwise complex isomorphisms (see~\cite{muniplst}). 
 Then we can construct the bundle mapping 
$G^\C\colon T(T^{n-1}\times[0,1])\otimes\C\to\C^n$ as 
\begin{equation*}
  G^\C_{(x,t)}:=
  \begin{cases}
    \ \lft(df_0^\C\rgt)_{(x,c(t))}=(df_0^\C)_{(x,\delta)}=(G_0^\C)_{(x,t)} 
    &(0\le t\le\epsilon)\\
    \ \lft(df_t^\C\rgt)_{(x,(1-c(t))\delta+c(t)(1-\delta)} 
    &(\epsilon\le t\le1-\epsilon) \\
    \ \lft(df_1^\C\rgt)_{(x,c(t))}=(df_1^\C)_{(x,1-\delta)}=(G_1^C)_{(x,t)} 
    &(1-\epsilon\le 1),
  \end{cases}
\end{equation*}
where $c\colon[0,1]\to[0,1]$ is some smooth function 
satisfying the following conditions for some small $\epsilon>0$: 
(i)\ $c(t)=0$, for $t\in[0,\epsilon]$, 
(ii)\ $c(t)=1$, for $t\in[1-\epsilon,1]$. 
 We should mention that the mapping $G^\C$ is fiberwise complex isomorphic. 
 Remark that $G^\C$ coincides 
with $G_0^\C$ on a neighborhood of $T^{n-1}\times\{0\}\subset T^{n-1}\times[0,1]$,
and with $G_1^\C$ on a neighborhood of $T^{n-1}\times\{1\}$. 

  By taking the real part of $G^\C$, we have a bundle mapping 
\begin{equation} \label{eq:lagmono}
  \begin{array}{rccc}
    G\colon &T(T^{n-1}\times[0,1]) &\longrightarrow &\xi \\
    &\downarrow & &\downarrow \\
    &T^{n-1}\times[0,1] &\longrightarrow &M. 
  \end{array}
\end{equation}
 It is a Lagrangian monomorphism that coincides 
with $df_0$ and $df_1$ on a neighborhood of 
$T^{n-1}\times\{0\}$, $T^{n-1}\times\{1\}\subset T^{n-1}\times[0,1]$, 
respectively.

  Then we apply the Smale-Hirsch immersion theorem 
(Theorem~\ref{thrm:shimmthm}). 
 On account of Theorem~\ref{thrm:shimmthm}, 
the existence of the bundle monomorphism $G$ implies 
the existence of an immersion $g\colon T^{n-1}\times[0,1]\to M$ 
for which $dg\colon T(T^{n-1}\times[0,1])\to TM$ is homotopic to $G$. 
 In addition, we may assume that $g=f_0$, $f_1$ 
on neighborhoods of $T^{n-1}\times\{0\}$, 
$T^{n-1}\times\{1\}\subset T^{n-1}\times[0,1]$, respectively. 
 Furthermore, by a slight perturbation of $g$ as immersions, 
we obtain an embedding $\tilde{g}\colon T^{n-1}\times[0,1]\to M$. 
 In fact, the dimension of $T^{n-1}\times[0,1]$ is $n$, that of $M$ is $2n+1$, 
and $n>1$. 
 By taking the perturbation fixing near the end $T^{n-1}\times\{0,\ 1\}$, 
we still may assume that $\tilde{g}=f_0$, $f_1$ 
on neighborhoods of $T^{n-1}\times\{0\}$, 
$T^{n-1}\times\{1\}\subset T^{n-1}\times[0,1]$, respectively. 
 Then, setting $\tilde{g}_t:=\tilde{g}(\cdot,t)$, 
we obtain a family $\tilde{g}_t\colon T^{n-1}\to M$ of embeddings 
which satisfies $\tilde{g}_0=f_0^\textup{cr}$ and $\tilde{g}_1=f_1^\textup{cr}$. 

  The obtained isotopy $\tilde{g}_t\colon T^{n-1}\times[0,1]\to M$ 
is what we need to apply Gromov's \textit{h}-principle. 
 In fact, setting $G_t(\cdot):=G(\cdot,t)|_{T_{(\cdot,t)}(T^{n-1}\times\{t\})}$, 
we have a family bundle mapping
\begin{equation*}
  \begin{array}{rccc}
    G_t\colon &T(T^{n-1}) &\longrightarrow &\xi \\
              &\downarrow & &\downarrow \\
              &T^{n-1} &\overset{\tilde{g}_t}{\longrightarrow} &M. 
  \end{array}  
\end{equation*}
 It is a family of isotropic monomorphisms covering $\tilde{g}_t$ 
which satisfies $G_t=df_0|_{T(T^{n-1}\times\{t\})}$ for $t\in[0,1]$ close to $0$, 
and $G_t=df_1|_{T(T^{n-1}\times\{t\})}$ for $t\in[0,1]$ close to $1$, 
and that $G_t$ is homotopic to $d\tilde{g}_t$. 
 In other words, it is a formal isotopy 
between subcritical isotropic embeddings 
$f_0^\textup{cr},\ f_1^\textup{cr}\colon T^{n-1}\to(M,\xi)$. 

  Now, we apply Gromov's \textit{h}-principle (Theorem~\ref{thrm:g_h-pr}) 
to obtain a path connecting subcritical isotropic embeddings 
$f_0^\textup{cr},\ f_1^\textup{cr}$. 
 By Theorem~\ref{thrm:g_h-pr}, we have a family
\begin{equation*}
  \tilde{f}_t\colon T^{n-1}\to(M,\xi)
\end{equation*}
of subcritical isotropic embeddings that satisfies 
$\tilde{f}_0=f_0^\textup{cr}$, $\tilde{f}_1=f_1^\textup{cr}$. 

%
\noindent\underline{\textbf{Step~2:}}\ 
  Next, we extend $\tilde{f}_t$ to a family of Legendrian embeddings 
of $T^{n-1}\times[0,1]$. 
 First, we endow $\tilde{f}_t$ with symplectically normal ``framings''. 
 Then, using the framings, we extend $\tilde{f}_t$ to Legendrian embeddings 
of $T^{n-1}\times[0,1]$. 

  As the framings, we construct a family $X_t$ 
of nowhere vanishing vector fields 
defined along isotropic submanifolds $\tilde{f}_t(T^{n-1})\subset(M,\xi)$ 
which are tangent to the conformal symplectic normal bundle 
$\sn\lft(\tilde{f}_t(T^{n-1}),M\rgt)$. 
 The conformal symplectic normal bundle 
is defined as 
\begin{equation*}
  \sn\lft(\tilde{f}_t(T^{n-1}),M\rgt)
  :=\lft.T\{\tilde{f}_t(T^{n-1})\}^{\skor}\rgt/T\{\tilde{f}_t(T^{n-1})\}\subset\xi,
\end{equation*}
where the symbol $\skor$ stands for the skew orthogonal subspace 
with respect to the symplectic form $d\alpha$ on $\xi=\ker\alpha$. 
 In other words, $X_t$ are chosen so that $X_t\oplus T\{\tilde{f}_t(T^{n-1})\}$ 
are Legendrian. 

  Around the ends $\{0,1\}\subset[0,1]$, 
we already have such framings. 
 In fact, the original embeddings $f_0,\ f_1\colon T^{n-1}\times[0,1]\to(M,\xi)$ 
are Legendrian. 
 Then the vector field $df_i(\uvv s)$, $i=0,1$, for coordinate $s\in[0,1]$ 
are the required framings. 
 Let $X_i$, $i=0,1$, denote such framings. 

  We construct a family $X_t$ of framings connecting $X_0$ and $X_1$ as follows.
 Recall that we have constructed a Lagrangian monomorphism 
$G\colon T(T^{n-1}\times[0,1])\to\xi$ (see Equation~\eqref{eq:lagmono}). 
 Considering $df_i$ on a neighborhood 
of $T^{n-1}\times\{i\}\subset T^{n-1}\times[0,1]$, $i=0,1$ respectively, 
we have a family $F_t\colon T(T^{n-1}\times[0,1])\to\xi$ 
of Legendrian monomorphisms 
that connects $df_0=F_0$ and $df_1=F_1$. 
 Setting $X_t:=F_t(\uvv s)$, we obtain a required family $X_t$ of vector fields,
where $s$ is a coordinate of $[0,1]$. 

  From the isotropic isotopy $\tilde{f}_t\colon T^{n-1}\to(M,\xi)$ 
with the isotropic framings $X_t$, 
we obtain a family $\bar{f}_t\colon T^{n-1}\times[0,1]\to(M,\xi)$ 
of Legendrian embeddings that satisfies 
$\bar{f}_i(T^{n-1}\times[0,1])=\Lambda_i$ for $i=0,1$. 
 In fact, the the given Legendrian submanifolds $\Lambda_i$ 
can be shrunk sufficiently to the core tori $f_i(T^{n-1})$ 
by a Legendrian isotopy along isotropic vector fields $(f_i)_\ast(\uvv s)$. 
 Then the family of embeddings of $T^{n-1}\times[0,1]$ constructed 
by $\tilde{f}_t(T^{n-1})$ and the framings $X_t$ implies 
the isotopy of embeddings connecting the shrunk $\Lambda_i$. 
 Thus, we obtain a family $\bar{f}_t\colon T^{n-1}\times[0,1]\to(M,\xi)$ 
of Legendrian embeddings that satisfies 
$\bar{f}_i(T^{n-1}\times[0,1])=f_i(T^{n-1}\times[0,1])=\Lambda_i$ for $i=0,1$. 

%
\noindent\underline{\textbf{Step~3:}}\ 
  Last of all we extend the isotopy $\bar{f}_t$ of Legendrian embeddings 
to an ambient isotopy. 

  We first review the ambient isotopy theorem (see for example \cite{geitext}). 
 It implies that an isotopy of isotropic embeddings can be extended 
to an isotopy of global contact diffeomorphisms. 
%
%
\begin{prop}\label{prop:amisotthm}
   Let $\psi_t\colon N\to(M,\xi)$ be an isotopy of isotropic embeddings 
  of a closed manifold $N$ into a contact manifold $(M,\xi)$. 
   Then there exists a compactly supported global contact isotopy 
  $\Psi_t\colon(M,\xi)\to(M,\xi)$ 
  that restricts to the given isotopy $\psi_t$. 
   In other words, $\Psi_t$ satisfies $\Psi_0=\textup{id}$, 
  and $\Psi_t\circ\psi_0=\psi_t$. 
\end{prop} 

  Applying the ambient isotopy theorem, we obtain the isotopy $\vphi_t$ 
required in the statement of Lemma~\ref{lem:key}. 
 In fact, we have constructed an isotopy 
$\bar{f}_t\colon T^{n-1}\times[0,1]\to(M,\xi)$ of Legendrian embeddings 
which satisfies $\bar{f}_0=f_0$ and $\bar{f}_1=f_1$. 
 Then, applying Proposition~\ref{prop:amisotthm}, 
we obtain a compactly supported  ambient contact isotopy 
$\vphi_t\colon(M,\xi)\to(M,\xi)$ 
that satisfies $\vphi_0=\textup{id}$, $\vphi_t\circ\bar{f}_0=\bar{f}_t$. 
 Especially, since 
$\bar{f}_i(T^{n-1}\times[0,1])=f_i(T^{n-1}\times[0,1])=\Lambda_i$ for $i=0,1$, 
it follows that $\vphi_1(\Lambda_0)=\Lambda_1$. 
 Further, since $\tilde{f}_t(T^{n-1})\subset(M,\xi)$ does not intersect 
the plastikstufe $\mathcal{P}$, so does $\bar{f}_t(T^{n-1}\times[0,1])$ 
for any $t\in[0,1]$. 
 Then $\vphi_t$ is constructed so that it is identity near $\mathcal{P}$. 
 Thus, the isotopy $\vphi_t$ is the required one. 
 Lemma~\ref{lem:key} has been proved. 
\end{proof}

  Now, we show Theorem~\ref{thrm:pstc2loose}. 
\begin{proof}[Proof of Theorem~\ref{thrm:pstc2loose}]
  The fundamental idea is to apply Theorem~\ref{thrm:slidingOTD} 
to a family of isotropic curves simultaneously. 
 In order to do that, we should move a Legendrian submanifold 
to a suitable position. 
 We need Lemma~\ref{lem:key} for that. 

  First, we move the given Legendrian submanifold by Lemma~\ref{lem:key}. 
 Let $\Lambda\subset(M,\xi)$ be a Legendrian submanifold 
disjoint from the plastikstufe $\mathcal{P}$. 
 Recall that the plastikstufe $\mathcal{P}\subset(M,\xi)$ 
has toric core with trivial rotation from the assumption. 
 Then, from Lemma~\ref{lem:key}, 
we have an isotopy $\vphi_t\colon(M,\xi)\to(M,\xi)$ 
and a Legendrian submanifold $\Lambda_0\subset\Lambda$ 
for which $\vphi_1(\Lambda_0)$ lies in the standard tubular neighborhood 
$U_\textup{PS}\subset(M,\xi)$ of $\mathcal{P}$ 
and diffeomorphic to $I\times T_0^{n-1}\subset\R^3\times T^\ast T^{n-1}$ 
by $\vphi_\textup{PS}\colon(U_\textup{PS},\xi)
\to\lft(\R^3\times T^\ast T^{n-1},\xi_\textup{PS}\rgt)$, 
where $I=(0,\epsilon)\subset(\R^3,\xi_\ot)$ is a Legendrian open segment 
and $T_0^{n-1}\subset T^\ast T^{n-1}$ is the zero section. 
 We may use the same notation $\Lambda,\Lambda_0$ 
for the modified Legendrian submanifolds 
$\vphi_1(\Lambda), \vphi_1(\Lambda_0)\subset(M,\xi)$. 

  Next, we apply an idea of Theorem~\ref{thrm:slidingOTD} 
to the Legendrian submanifold $\Lambda$. 
 We discuss in the standard tubular neighborhood $U_\textup{PS}\subset(M,\xi)$ 
of the plastikstufe $\mathcal{P}$ with the contact embedding 
$\vphi_\textup{PS}\colon(U_\textup{PS},\xi)
\to\lft(\R^3\times T^\ast T^{n-1},\xi_\textup{PS}\rgt)$. 
 By taking $U_\textup{PS}$ appropriately small, 
we may assume $\Lambda\cap U_\textup{PS}=\Lambda_0$. 
 In other words, we discuss in $\lft(\vphi_\textup{PS}(\ups),\xi_\textup{PS}\rgt)
\subset\lft(\R^3\times T^\ast T^{n-1},\xi_\textup{PS}\rgt)$ 
and use the same notation $\Lambda_0$ for 
$\vphi_\textup{PS}(\Lambda_0)
\subset\lft(\R^3\times T^\ast T^{n-1},\xi_\textup{PS}\rgt)$. 
 Recall that in the standard neighborhood, the plastikstufe $\mathcal{P}$ 
is $D_\ot^2\times T_0^{n-1}\subset\R^3\times T^\ast T^{n-1}$ 
and that the Legendrian submanifold $\Lambda_0$ is disjoint from $\mathcal{P}$. 
 From the discussion above, $\Lambda_0$ is 
$I\times T_0^{n-1}\subset\R^3\times T^\ast T^{n-1}$. 
 Each object $\mathcal{P}$, $\Lambda_0$ is a direct product, 
with the zero-section $T_0^{n-1}\subset T^\ast T^{n-1}$, 
of an overtwisted disc 
$D_\ot^2\subset(\R^3,\ker\alpha_\ot)$ 
and a Legendrian segment $I\subset(\R^3,\ker\alpha_\ot)$, respectively. 
 Then, for each $p\in T_0^{n-1}$, we apply Theorem~\ref{thrm:slidingOTD} 
simultaneously. 
 We obtain a Legendrian open segment 
$\tilde{I}\subset(\R^3,\ker\alpha_\ot)$ 
which is a negative destabilization of $I$, 
and a Legendrian submanifold 
$\tilde{I}\times T_0^{n-1}\subset(\R^3\times T^\ast T^{n-1},\xi_\textup{PS})$. 
 Let $\tilde{\Lambda}_0$ denote $\tilde{I}\times T_0^{n-1}$ 
or the corresponding Legendrian submanifold in $(M,\xi)$. 
 Comparing $\Lambda_0\subset\Lambda$ and $\tilde\Lambda_0$, 
we look for a Loose chart for $\Lambda\subset(M,\xi)$. 

  We look for the loose chart 
in $\lft(\vphi_\textup{PS}(\ups),\xi_\textup{PS}\rgt)
\subset\lft(\R^3\times T^\ast T^{n-1},\xi_\textup{PS}\rgt)$. 
 Let $U\subset(\R^3,\xi_\ot=\ker\alpha_\ot)$ 
be an open ball including $I,\tilde{I}$, 
and $D^\ast T^{n-1}\subset T^\ast T^{n-1}$ a disc-bundle for some metric 
that satisfy $U\times D^\ast T^{n-1}\subset\vphi_\textup{PS}(U_\textup{PS})$. 
 We find appropriate subsets in both $U$ and $D^\ast T^{n-1}$ 
in what follows. 

  We find a neighborhood of $I\subset U\subset(\R^3,\xi_\ot)$ 
contactomorphic to a convex open subset in $(\R^3,\xi_\textup{std})$ 
where $I$ is a negative stabilization as follows. 
 From the view point of $\tilde{I}$, 
there exists a tubular neighborhood $\tilde{W}\subset(U,\xi_\ot)$,
including $I$ as well after some isotopy, 
which is contactomorphic to the standard tubular neighborhood $W$ of 
$\{x=0,z=0\}\subset\lft(\R^3,\xi_\textup{std}=\ker(dz-ydx)\rgt)$. 
 Note that the image of $I$ in $W\subset(\R^3,\xi_\textup{std})$ is 
a negative stabilization of the image of $\tilde{I}$. 
 We abuse the same notation $I,\tilde{I}$ even in $(\R^3,\xi_\textup{std})$. 
 Then 
$W\times D^\ast T^{n-1}\subset\lft(\R^3\times T^\ast T^{n-1},\xi_\textup{st}\rgt)$ 
is a tubular neighborhood of $\tilde{I}\times T_0^{n-1}$
which is contactomorphic to a neighborhood 
of $\tilde{\Lambda}_0=\tilde{I}\times T_0^{n-1}\subset(\ups,\xi)$. 

  Further, we find a sufficiently large, in some sense,  subspace 
in $\lft(D^\ast T^{n-1},\lambda_\textup{can}\rgt)$. 
 For any metric, the set $V_\rho:=\{(\mathbf{p},\mathbf{q})\in T^\ast T^{n-1}\mid 
|\mathbf{p}|<\rho,|\mathbf{q}|<\rho\}$ is included in $D^\ast T^{n-1}$ 
for some $\rho>0$. 
 Then take a constant $\epsilon>0$ so that $\epsilon<\rho$. 
 For the constant $\epsilon$, 
let $L_\epsilon$ be a Legendrian curve in $(\R^3,\xi_\textup{std})$ defined as 
\begin{equation*}
  \lft(\epsilon t^2,\dfrac{15}{4}\epsilon,\dfrac{\epsilon^2}{2}(3t^2-5t^3)\rgt),
\end{equation*}
that is isotopic to a negative stabilization 
of $\{x=0,z=0\}\subset(\R^3,\xi_\textup{std})$. 
 Since $L_\epsilon$ is isotopic to $I$, 
$\lft(W\times D^\ast T^{n-1},L_\epsilon\times T_0^{n-1}\rgt)$ is contactomorphic to 
$\lft(W\times D^\ast T^{n-1},\Lambda_0=I\times T_0^{n-1}\rgt)$. 

  Then the pair $\lft(N,\tilde{\Lambda}_0'\rgt)$ of submanifolds of $(M,\xi)$, 
corresponding to $\lft(W\times V_\rho,L_\epsilon\times T_0^{n-1}\rgt)$ 
in $\lft(\R^3\times T^\ast T^{n-1},\xi_\textup{std}\rgt)$ 
by $\Psi\circ\vphi_\textup{PS}$, is contactomorphic to a loose chart. 
 In fact, by the contactomorphism 
\begin{equation*}
  h\colon(\R^3,\xi_\textup{std})\to(\R^3,\xi_\textup{std}),\quad 
  (x_i,y_i,z)\mapsto(x_i/\epsilon,y_i/\epsilon,z/\epsilon^2), 
\end{equation*}
$W\times V_\rho$ and $L_\epsilon\times T_0^{n-1}$ is mapped 
to $W'\times V_{\rho/\epsilon}$ and $L'_\epsilon\times T_0^{n-1}$ respectively, 
where $W'\subset\R^3$ is an open set, $L'_\epsilon\subset(\R^3,\xi_\textup{std})$
is a Legendrian segment isotopic to $L_\epsilon$. 
 Since $\epsilon<\rho$, we have $\rho/\epsilon>1$. 

  Thus, Theorem~\ref{thrm:pstc2loose} has been proved. 
\end{proof}

\subsubsection{Proof of overtwistedness}\label{sec:pfps2pt}
  Now, we show the main issue of this subsection. 
 We show that the existence of a small plastikstufe 
with toric core and trivial rotation 
implies overtwistedness. 
%
%
\begin{prop}\label{prop:suff}
  Let $(M,\xi)$ be a contact manifold of dimension~$2n+1>3$. 
  The contact structure $\xi$ is overtwisted 
  if there exists a small plastikstufe with toric core 
  that has trivial rotation. 
\end{prop} 

\begin{proof}
  Assume that there exists a small plastikstufe
$\mathcal{P}_{T^{n-1}}$ with toric core and trivial rotation 
in a contact manifold $(M,\xi)$ of dimension~$2n+1$. 
 Let $\Lambda_0\subset(M,\xi)$ be the trivial Legendrian sphere 
(see Subsection~\ref{sec:overtwisted}). 
 Applying some isotopy, we may assume 
that $\Lambda_0\cap\mathcal{P}_{T^{n-1}}=\emptyset$. 
 Then, from Theorem~\ref{thrm:pstc2loose}, 
the Legendrian submanifold $\Lambda_0$ is loose. 
 By means of Proposition~\ref{prop:ls2ot}, 
the contact structure $\xi$ is overtwisted.
\end{proof}

%

\section{Modification that creates plastikstufes with toric core}
\label{sec:gltw}
  The motivation of this paper is a generalized Lutz twist 
introduced in~\cite{art21}. 
 The modification makes a contact manifold overtwisted in any dimension. 
 We should note that it is proved directly in~\cite{art21} 
that the modification makes an $S^1$-family of overtwisted discs. 
 On the other hand, using a result in this paper, 
it is confirmed that the modification makes a contact manifold overtwisted 
from another point of view. 
 In this section, we find a small plastikstufe 
with toric core that has trivial rotation in the modification. 

  First, we roughly recall a generalized Lutz twist introduced in~\cite{art21}. 
 Let $(M,\xi)$ be a contact manifold of dimension~$2n+1$, 
and $\Gamma\subset(M,\xi)$ an embedded circle transverse to $\xi$. 
 Along $\Gamma$ there exists the standard tubular neighborhood 
contactomorphic to a tubular neighborhood of the transverse circle 
$S^1\times\{0\}\subset(S^1\times\R^{2n},\xi_0=\ker\alpha_0)$, where 
\begin{equation*}
  \alpha_0=d\phi+\sum_{i=1}^n r_i^2d\theta_i
\end{equation*}
with respect to the cylindrical coordinates $(\phi,r_i,\theta_i)$ 
of $S^1\times\R^{2n}$. 
 The generalized Lutz twist is defined as a replacement of 
such a small tubular neighborhood 
with the same neighborhood with a certain contact structure, 
without changing the manifold. 
 The substitute is constructed in~\cite{art21}. 
 Leaving the precise definition to~\cite{art21}, 
we introduce an important part for the construction of the substitute, 
where we find plastikstufes. 
 Let $\zeta$ be the hyperplane field on $S^1\times\R^{2n}$
defined as $\zeta=\ker\omtw$, where
\begin{equation*}
  \omtw
  =\prod_{i=1}^n(\cos r_i^2)d\phi
  +\sum_{i=1}^n(\sin r_i^2)d\theta_i, 
\end{equation*}
with respect to the cylindrical coordinates $(\phi,r_i,\theta_i)$ 
of $S^1\times\R^{2n}$. 
 Unfortunately, it is not a contact structure but a confoliation. 
 In other words, it has a certain non-contact locus. 
 However, it is proved in~\cite{art21} 
that it can be approximated to a contact structure. 

  It is proved, in~\cite{art21}, 
that there exists an $S^1$-family of plastikstufes with toric core 
in $(S^1\times\R^{2n},\zeta)$. 
 Although it is described, in \cite{art21}, in terms of the so-called 
bordered Legendrian open book (bLob), 
a generalization of plastikstufe, 
it can be translated easily. 
 The object to be observed is the submanifold 
\begin{equation*}
  P:=\lft\{(\phi,r_1,\theta_1,\dots,r_n,\theta_n)\in S^1\times\R^{2n}\;\lft|\;
  r_1\le\sqrt{\pi},\ r_2=\dots=r_n=\sqrt{\pi}\rgt.\rgt\}. 
\end{equation*}
 It is diffeomorphic to $S^1\times D^2\times T^{n-1}$. 
 From \cite{art21}, $P$ lies in the contact locus of $(S^1\times\R^{2n},\zeta)$ 
and $P$ is an $S^1$-family of bLobs whose bindings are $T^{n-1}$. 
 We show that each 
\begin{align}\label{eq:psparam}
  P_s:=&\lft\{(\phi,r_i,\theta_i) \in S^1\times\R^{2n}\;
         \lft|\; \phi=s,r_1\le\sqrt{\pi},r_2=\dots=r_n=\sqrt{\pi}\rgt.\rgt\}, 
         \qquad s\in S^1, \notag \\
  =&\lft\{(s,r_1,\theta_1,\sqrt{\pi},\theta_2,\dots,\sqrt{\pi},\theta_n)
     \in S^1\times\R^{2n}\;\left|\; r_1\in[0,\sqrt{\pi}],\theta_1,\dots,\theta_n
     \in S^1\rgt.\rgt\}
\end{align}
is a plastikstufe with toric core $T^{n-1}$ 
in terms of the definition of plastikstufe in Subsection~\ref{sec:plastikstufe}.
 In fact, $P_s$ is diffeomorphic to $D^2\times T^{n-1}$. 
 With respect to this correspondence, 
a point $z\in D^2$ corresponds to $(r_1^z,\theta_1^z)\in\R^2$. 
 Then the submanifold 
\begin{equation*}
  \{z\}\times T^{n-1}
  =\lft\{(s,r_1^z,\theta_1^z,\sqrt{\pi},\theta_2,\dots,\sqrt{\pi},\theta_n)
  \in S^1\times\R^{2n}\mid \theta_2,\dots,\theta_n\in S^1\rgt\}\subset P_s
\end{equation*}
is tangent to $\zeta=\ker\omtw$, 
since $\omtw|_{T(\{z\}\times T^{n-1})}=(-1)^{n-1}\cos (r_i^z)^2d\phi$. 
 On the other hand, a point $b\in T^{n-1}$ corresponds to 
$\lft(\sqrt{\pi},\theta_2^b,\dots,\sqrt{\pi},\theta_n^b\rgt)\in\R^{2n-2}$. 
 Then on the submanifold 
\begin{equation*}
  D^2\times\{b\}=\lft\{(s,r_1,\theta_1,\sqrt{\pi},\theta_2^b,\dots,\sqrt{\pi}, 
  \theta_n^b)
  \in S^1\times\R^{2n}\mid r_1\in(0,1],\theta_1\in S^1\rgt\}\subset P_s, 
\end{equation*}
the $1$-form restricts to $\omtw|_{T(D^2\times\{b\})}=(\sin r_1^2)d\theta_1$. 
 Then the singular foliation on $D^2\times\{b\}$ 
generated by $\zeta\cap T(D^2\times\{b\}$ is like Figure~\ref{fig:3otdiscs}~(2).
 Thus, we have confirmed that $P_s$ is a plastikstufe. 

  It is clear that the plastikstufe $P_s\subset S^1\times\R^{2n}$ is ``small.'' 
 In fact, $P_s$ lies in the boundary of the polydisc 
$\{s\}\times (D^2)^n\subset S^1\times\R^{2n}$. 
 And the entire polydisc is taken in, 
in the modification defined in~\cite{art21}. 

  It remains to show that the plastikstufe has trivial rotation. 
 According to the Definition~\ref{df:rtcl} in Subsection~\ref{sec:torictrivial},
we show that the leaf ribbon of the Plastikstufe $P_s$ 
is Legendrian isotopic to the standard Legendrian $(0,1)\times T^{n-1}$ 
in $(S^1\times\R^{2n},\zeta)$. 
 A leaf ribbon of $P_s$ is 
\begin{equation*}
  LR_{\theta_1^c}
  :=\lft\{(s,r_1,\theta_1^c,\sqrt{\pi},\theta_2,\dots,\sqrt{\pi},\theta_n)
  \in S^1\times\R^{2n}\; \lft|\; 
  r_1\in(0,\sqrt{\pi}), \theta_2,\dots,\theta_n\in S^1\rgt.\rgt\}
  \cong (0,1)\times T^{n-1} 
\end{equation*}
for a constant $\theta_1^c\in S^1$ 
comparing it with $P_s$ in Equation~\eqref{eq:psparam}. 
 Note that $\lan\uvv{r_1},\uvv{r_2}\ran$ generates an isotropic plane field 
for $\zeta=\ker\omtw$. 
 Then $LR_{\theta_1^c}$ is Legendrian isotopic to
\begin{align*}
  \widetilde{LR}_{\theta_1^c}
  :=&\lft\{(s,\sqrt{\pi},\theta_1^c,r_2,\theta_2,
      \sqrt{\pi},\theta_3\dots,\sqrt{\pi},\theta_n)
      \in S^1\times\R^{2n}\; \lft|\; 
      r_2\in(0,\sqrt{\pi}), \theta_2,\dots,\theta_n\in S^1\rgt.\rgt\} \\
  \cong& (0,1)\times T^{n-1}. 
\end{align*}
  By regarding 
$\lft\{(s,\sqrt{\pi},\theta_1^c,r_2,\theta_2,(0,0),\dots,(0,0))
\in S^1\times\R^{2n}\;
\lft|\;r_2\in[0,\sqrt{\pi}],\theta_2\in S^1\rgt.\rgt\}=D^2$, 
the modified leaf ribbon $\widetilde{LR}_{\theta_1^c}$ can be considered as 
\begin{equation*}
  \{s\}\times\lft\{(\sqrt{\pi},\theta_1^c)\rgt\}
  \times \lft(D^2\setminus\{0\}\rgt)\times T^{n-2}
  \subset\lft(S^1\times\R^2\times\R^2\times\R^{2(n-2)},\zeta\rgt),  
\end{equation*}
the standard embedded Legendrian $(0,1)\times T^{n-1}$. 
 This implies that the plastikstufe $P_s$ has trivial rotation. 

%
%

\medskip 

\begin{flushleft}
Department of Mathematics, \\
Hokkaido University, \\
Sapporo, 060--0810, Japan. \\
\medskip
e-mail: j-adachi@math.sci.hokudai.ac.jp
\end{flushleft}
\end{document}